\documentclass{amsart}
\usepackage[utf8]{inputenc}
\usepackage{amsthm}
\usepackage{amsrefs}

\theoremstyle{plain}
\newtheorem{theorem}{Theorem}
\newtheorem{proposition}{Proposition}
\newtheorem{corollary}{Corollary}
\newtheorem{lemma}{Lemma}
\theoremstyle{definition}
\newtheorem{definition}{Definition}
\theoremstyle{remark}
\newtheorem{remark}{Remark}
\newtheorem{example}{Example}

\newcommand*\derS{\mathrm{der}}

\begin{document}
\title{Universally uniformly continuous metric spaces}

\author{Katrina Gensterblum}
\address{Michigan State University, East Lansing, Michigan, USA}
\email{genster6@msu.edu}
\author{Peikai Qi}
\address{Nankai University, Tianjin, PRC}
\email{18096781997@189.cn}
\author{Willie Wong}
\address{Michigan State University, East Lansing, Michigan, USA}
\email{wongwwy@member.ams.org}


\begin{abstract}
	We answer the question: ``on which metric spaces $(M,d)$ are all continuous functions uniformly continuous?'' Our characterization theorem improves and generalizes a previous result due to Levine and Saunders, and in particular is applicable to metric spaces which are ``infinite dimensional.''
\end{abstract}

\maketitle

\section{Introduction}

It is well-known that if a set is compact then every continuous function on it is uniformly continuous. 
The converse however is not true. 
Previously, Levine \cite{Levine}, and later Levine and Saunders \cite{LS}, began to investigate sets on which every continuous function is uniformly continuous. 
In this paper we describe sufficient and necessary conditions for such sets (in the category of metric spaces), bringing their program to a conclusion. 
Before we give more technical descriptions, we'll introduce some important definitions.

\begin{definition}
	A metric space $(M,d)$ is said to be \emph{uniformly continuous} if every real-valued continuous function on $M$ is uniformly continuous. More generally, $(M,d)$ is said to be uniformly continuous with respect to a metric space $(M_1, d_1)$ if every continuous function from $M$ to $M_1$ is uniformly continuous.
\end{definition}

\begin{definition}
A metric space $(M,d)$ is said to be \emph{universally uniformly continuous} if it is uniformly continuous with respect to every metric space.
\end{definition}

\begin{definition}
A subset $E$ of a metric space $(M,d)$ is said to be \emph{uniformly isolated} if and only if there exists a real number $\eta > 0$ such that $d(p,q) \geq \eta$ for every $p \neq q$ in $E$.
\end{definition}

N. Levine in \cite{Levine} studied uniformly continuous subsets of the real line. 
He proved that the union of a compact set and an uniformly isolated set is uniformly continuous; and conversely, that every uniformly continuous subset of the real line is equal to the union of a compact set and an uniformly isolated set. 
Levine and Saunders \cite{LS} extended the first statement for general metric spaces; in particular they showed that if a metric space admits a decomposition as the union of a compact subset and a uniformly isolated subset, then it is uniformly continuous. 
However, they can only obtain a partial converse under the further assumption that the metric space has the so-called WB property. (A metric space has the WB property if the Heine-Borel theorem holds.) 

In this paper we give an unconditional characterization of uniformly continuous metric spaces. 
First we give a counterexample to the na\"ive conjecture that uniformly continuous spaces can always be decomposed into the union of a compact set and an uniformly isolated set.

\begin{example}
Take the set $M$ of infinite sequences, where a point $x$ in $M$ has the form $x=(x_1,x_2,x_3,\cdots ,x_n,\cdots)$. For $x,y \in M$, where $x=(x_1,x_2,x_3,\cdots ,x_n,\cdots)$ and $y =(y_1,y_2,y_3,\cdots,y_n,\cdots)$ we define the distance $d(x,y)$ as:
\[d(x,y):=\sqrt{(x_1-y_1)^2+(x_2-y_2)^2+\cdots +(x_n-y_n)^2+\cdots}.\]
Now let $E_0=\{p_{ij}\mid i,j \in \mathbb{Z}_{> 0}\} \subset M$ and $p_{ij}=(0,0,\cdots,\frac{1}{j},0,0,\cdots)$ where the $\frac{1}{j}$ is in the $i$th-slot. In other words,
\[ E_{0} = 
\begin{Bmatrix}
  (1,0,0,0,\cdots), & (0,1,0,0,\cdots), & (0,0,1,0,\cdots), & (0,0,0,1,\cdots), & \cdots \\
  (\frac{1}{2},0,0,0,\cdots), &(0,\frac{1}{2},0,0,\cdots), &(0,0,\frac{1}{2},0,\cdots),&(0,0,0,\frac{1}{2},\cdots), & \cdots \\
   (\frac{1}{3},0,0,0,\cdots), & (0,\frac{1}{3},0,0,\cdots), & (0,0,\frac{1}{3},0,\cdots), &(0,0,0,\frac{1}{3},\cdots),  & \cdots\\
    (\frac{1}{4},0,0,0,\cdots), & (0,\frac{1}{4},0,0,\cdots), & (0,0,\frac{1}{4},0,\cdots), &(0,0,0,\frac{1}{4},\cdots),  & \cdots\\
  \vdots& \vdots & \vdots & \vdots &\ddots
\end{Bmatrix}.
\]
Let $E=E_0\bigcup \{\vec{0}\}$. In this case $E$ fails the WB property, and $E$ does not equal the union of a compact set and an uniformly isolated set. However, $(E,d)$ is a uniformly continuous metric space.
\end{example}

\begin{proof}
We begin by observing that the subset $E\setminus B(\vec{0},\frac1m)$ is uniformly isolated: in fact, given any $x \in E\setminus B(\vec{0},\frac1m)$, and any $y \in E$, we can check that $d(x,y) \geq \frac{1}{m(m+1)}$. This allows us to prove by contradiction that $(E,d)$ is uniformly continuous. Suppose $f$ is a continuous, but not uniformly continuous function on $E$. Then there must exist $\varepsilon_0 > 0$ and two sequences $\{x_{(n)}\}$ and $\{y_{(n)}\}$ in $E$, satisfying $\lim\limits_{n\to\infty} d(x_{(n)}, y_{(n)}) = 0$, such that $|f(x_{(n)}) - f(y_{(n)})| > \varepsilon_0$ for every $n$. Using the uniform isolation of $E\setminus B(0,\frac1m)$ discussed above, we have that 
\[ d(x_{(n)}, y_{(n)}) < \frac{1}{m(m+1)} \implies x_{(n)}, y_{(n)} \in B(\vec{0}, \frac1m).\]
Hence $\lim\limits_{n\to \infty}x_{(n)}=\vec{0} = \lim\limits_{n\to \infty}y_{(n)}$. Then by the continuity of $f$ we must have $\lim\limits_{n\to\infty} f(x_{(n)}) = \lim\limits_{n\to\infty} f(y_{(n)})$, giving a contradiction.

Next we prove that $E$ is not the union of a uniformly isolated subset with a compact subset. Assume for contradiction that $E$ can be decomposed into $E_1$ and $E_2$, where $E_1$ is an uniformly isolated set and $E_2$ is a compact set.
(Note that as $E_2$ is compact, it is closed, so $E_1$ is open.)
First we can claim that, since $\vec{0}$ is a limit point of $E$, the set $E_1$ avoids an open ball centered at $\vec{0}$. For if not, there exists a sequence $\{x_{(n)}\}$ in $E_1$ satisfying $0 < d(\vec{0}, x_{(n)}) < \frac1n$, which implies $d(x_{(n)}, x_{(n+1)}) < \frac2n$, contradicting the uniformly isolated property.
This implies that for some positive integer $m$, the ball $B(\vec{0}, \frac{1}{m-1}) \subseteq E_2$. Consider the sequence $p_{1m}, p_{2m}, p_{3m}, \dots$ where $p_{ij}$ is as defined above. Each of the $p_{im} \in B(\vec{0}, \frac{1}{m-1}) \subseteq E_2$, but for $i\neq j$ we have $d(p_{im}, p_{jm}) = \sqrt{2}/m > 0$, hence we've found a bounded sequence in $E_2$ that does not admit a converging subsequence, contradicting the assumption that $E_2$ is compact. 
(This argument also shows that $E$ fails the WB property, as $B(\vec{0}, \frac{1}{m-1})$ is a closed and bounded subset of $E$ that is not compact.)
\end{proof}

\begin{remark}
$(E_0,d)$ is not an uniformly continuous metric space.  Take the bounded function
     \[f=\begin{cases}
                  1 & \text{if $i+j$ is even} \\
                  0 & \text{if $i+j$ is odd}
                \end{cases}
                \]
It is continuous as the topology induced on $(E_0,d)$ is discrete. 
It is not uniformly continuous as $d(p_{1j}, p_{1(j+1)}) = \frac{1}{j(j+1)}$ can be made arbitrarily small while $|f(p_{1j}- p_{1(j+1)}| = 1$.  
\end{remark}

It turns out that although an uniformly continuous metric space may not always admit a decomposition into a compact set and an uniformly isolated set, such a decomposition is ``almost'' available. 
This is the first of our main results. We can furthermore precisely identify the compact set in the decomposition.
Take our example above. $E_0$ is not uniformly continuous, but the set $E$ is. This is because $E$ contains $\{\vec{0}\}$, a compact subset which contains the unique limit point of the set. This motivates the following theorem:

\begin{theorem}\label{metricversion}
Let $(M,d)$ be a metric space. Then the following are equivalent:
\begin{itemize}
	\item $(M,d)$ is uniformly continuous.
    \item the derived set $\derS(M)$ is compact, and $\forall \varepsilon>0$, the set $M\setminus B(\derS(M),\varepsilon)$ is uniformly isolated.
\end{itemize}
\end{theorem}
\begin{remark}
	The set $\derS(M)$ is defined to be the set of all limit points of $(M,d)$. The set $B(\derS(M),\varepsilon)$ is the set of all points of distance at most $\varepsilon$ from $\derS(M)$; equivalently we take 
    \[ B(\derS(M), \varepsilon) = \bigcup_{x\in \derS(M)} B(x,\varepsilon).\]
    Note that when $\derS(M) = \emptyset$ this means that for every $\varepsilon > 0$, the set $M \setminus B(\derS(M),\varepsilon) = M$. 
\end{remark}

Our theorem does not refer to the WB condition of Levine and Saunders. 
This allows it to be applied to more general situations. 
A particular consequence of this is that our theorem applies also to \emph{universally} uniformly continuous metric spaces.

\begin{corollary}\label{basicCor} 
Let $(M,d)$ be a metric space. Then the following are equivalent: 
\begin{itemize}
	\item $(M,d)$ is universally uniformly continuous. 
    \item the derived set $\derS(M)$ is compact, and $\forall \varepsilon>0$, the set $M\setminus B(\derS(M),\varepsilon)$ is uniformly isolated.
\end{itemize}
\end{corollary}

It is natural to ask whether a similar characterization statement can be made using purely topological properties. The obvious answer is no, noticing that
uniform continuity is not a topological property. 
In particular, we can have two metrics $d_1, d_2$ on the same set $M$ generating the same topology where one is uniformly continuous, but the other is not. 
\begin{example}
Take $M$ to be the set of all positive integers, and let $d_1(a,b) = |a-b|$ be the distance function defined as normal, and set $d_2(a,b) = |\frac1a - \frac1b|$. 
By Theorem \ref{metricversion} $(M,d_1)$  is uniformly continuous, but not $(M,d_2)$. 
\end{example}
Our second result instead answers the refined question ``when does a metric space $(M,\delta)$ admit a (possibly different) metric $d$ generating the same topology, with $(M,d)$ being (universally) uniformly continuous?''

Notice that two metric spaces $(M, \delta)$ and $(M,d)$ with the same topology admit the same set of continuous functions $\mathcal{C}$. 
But with respect to the two distance functions the corresponding subsets $\mathcal{U}_\delta, \mathcal{U}_d \subset \mathcal{C}$ of uniformly continuous functions can be different. 
Our goal can be interpreted as asking when can one choose $d$ to maximize $\mathcal{U}_d$ so that it is in fact equal to $\mathcal{C}$. 

\begin{theorem}\label{topologyversion}
Let $(M,\delta)$ be metric space. If $\derS(M)$ is a compact set, then $\exists$ a new distance function $d$ on $M$ such that 
\begin{itemize}
	\item $(M,d)$ is uniformly continuous,
    \item $(M,\delta)$ and $(M,d)$ have the same open sets.
\end{itemize}
\end{theorem}

\begin{remark}
Equivalently, Theorem \ref{topologyversion} states that: if $M$ is a metrizable topological space whose derived set is compact, then there is a choice of metric making $M$ uniformly continuous.  
\end{remark}

\subsection*{Acknowledgements}
This paper is the result of an undergraduate research project conducted by the two first named authors at Michigan State University.
KG's participation is supported through the Professorial Assistantship program by the Michigan State University Honors College.
PQ's participation is supported by the exchange program organized by the MSU Math department.

After the initial arXiv posting of this manuscript it was pointed out to the authors that the results proven in the manuscript are not new; in fact, what we defined as ``uniformly continuous spaces'' now are properly known under the name ``Atsuji spaces'', and there exists an entire literature on their properties. 
The authors would like to thank Subiman Kundu for bringing this to their attention, and refer the readers to the exhaustive review article \cite{Kundu} for more information on these spaces. 

\section{Proof of Theorem \ref{metricversion} and its Corollary}
Throughout this section we fix $(M,d)$ to be a metric space. 
\begin{lemma}\label{lem:cpctnbhd}
Let $K\subset U \subset M$, where $K$ is nonempty compact and $U$ open. 
Then there is $r>0$ such that $B(z,r) \subset U $ for any $z \in K$. Note that $r$ doesn't depend on $z \in K$.
\end{lemma}

\begin{proof}
For each $z \in K \subset U$, since $U$ is open there exists $\delta_{z} > 0$ such that $B(z,\delta_z)\subset U$. Then $\{B(z,\delta_z/2)\mid z\in K\}$ is an open cover of $K$. Since $K$ is compact, there is a finite sub-cover associated to $\{z_1, \ldots, z_n\} \subset K$. 
Let $r=\min\{\delta_{z_1}/2,\delta_{z_2}/2,\dots,\delta_{z_n}/2\}$. 

Now let $z\in K$ be arbitrary. Our construction above implies that there is some $i\in \{1, \dots, n\}$ such that $d(z, z_i) < \delta_{z_i} / 2$. 
By the triangle inequality we see then any point $y\in B(z,r)$ is also in $B(z_i, \delta_{z_i}/2 + r) \subset B(z_i, \delta_{z_i}) \subset U$ as claimed. 
\end{proof}

\begin{proposition}\label{prop:derCUIUC}
If $\derS(M)$ is compact and $\forall \varepsilon > 0$ the set $M \setminus B(\derS(M),\varepsilon)$ is uniformly isolated, then $(M,d)$ is universally uniformly continuous. 
\end{proposition}

\begin{proof}
Let $(M', d')$ be an arbitrary metric space, and suppose $f: M \to M'$ is a continuous function. 
Fix $\eta > 0$, continuity implies that for every $x\in M$ there exists $\delta_x > 0$ such that for every $y\in B(x,\delta_x)$ we have $f(y) \in B(f(x), \eta / 2)$. 
By compactness of $\derS(M)$, we can find a finite subset $\{x_1, \dots, x_n\} \subset \derS(M)$ such that 
\[ \derS(M) \subset \bigcup_{i = 1}^n B(x_i, \frac{\delta_{x_i}}{3}) =: U_1. \]
Let us set $\delta_1 = \min_{i\in \{1, \ldots, n\}} \delta_{x_i}/3$. 
Now $U_1 \supset \derS(M)$, and by Lemma \ref{lem:cpctnbhd} there exists some $\delta_2 > 0$ such that $B(\derS(M),2\delta_2) \subset U_1$.  We can assume $2\delta_2 \leq \delta_1$, otherwise we can replace $\delta_2$ by $\delta_1 /2$  and the conclusion of the Lemma still holds.  
By hypothesis $M \setminus B(\derS(M), \delta_2)$ is uniformly isolated, so there exists $\delta_3$ such that if $z_1, z_2\in M\setminus B(\derS(M), \delta_2)$ then $d(z_1, z_2) > \delta_3$. We can assume $\delta_3 < \delta_2$; otherwise we can replace $\delta_3$ by $\delta_2$ and the conclusion still holds. 

Now suppose $z_1, z_2\in M$ are such that $d(z_1, z_2) < \delta_3$. 
Then necessarily at least one of the two points is in $B(\derS(M), \delta_2)$ and thus both are in $B(\derS(M), 2\delta_2) \subset U_1$. 
By construction then there exists some $j\in \{1, \dots, n\}$ such that $z_1 \in B(x_{j}, \frac{\delta_{x_j}}{3})$. And thus also $z_2 \in B(x_{j}, \frac{\delta_{x_j}}{3} + \delta_3)$. This implies that both $z_1, z_2 \in B(x_{j}, \delta_{x_j})$, by our construction of $\delta_1, \delta_2, \delta_3$. Therefore both $f(z_1), f(z_2) \in B(f(x_j), \eta / 2)$, which implies $d'(f(z_1), f(z_2)) < \eta$. And we have shown that $f$ is uniformly continuous. 
\end{proof}

This proves one direction of our desired result. 
For the reverse direction, first we recall Theorem 1 from \cite{LS}, which states, in our terminology, 
\begin{proposition}[Levine-Saunders] \label{prop:LS1}
	If $(M,d)$ is an uniformly continuous metric space, then $(M,d)$ is complete. 
\end{proposition}
The proof of this result makes use of the following construction, which we will also use frequently below, and hence we include a short proof. 
\begin{lemma}\label{lem:const}
	If $(M,d)$ is an uniformly continuous metric space, and $A, B\subset M$ are disjoint closed subsets, then $\inf_{x\in A, y\in B} d(x,y) > 0$. 
\end{lemma}
\begin{proof}
	Consider the function $f: x \mapsto d(x,A) / [d(x,A) + d(x,B)]$ on $M$. Noting that $d(x,A)$ vanishes precisely on the closure of $A$, we see that $d(x,A) + d(x,B) > 0$ for every $x$, as $A$ and $B$ are assumed to be disjoint and closed. Therefore $f$ is well-defined. 
    Secondly, it is clear by triangle inequality that the function $x\mapsto d(x,A)$ is continuous, therefore $f$, as a quotient of continuous functions, is continuous. 
    However, for every $x\in A$ we have $f(x) = 0$, and for every $y\in B$ we have $f(y) = 1$. So if $\inf_{x\in A, y\in B} d(x,y) = 0$, the function $f$ would be an example of a continuous but not uniformly continuous function. Our lemma follows. 
\end{proof}

\begin{proof}[Proof of Proposition \ref{prop:LS1}]
	We prove by contradiction. Let $(x_n)_{n\geq 1}$ be a Cauchy sequence in $(M,d)$. Suppose $(x_n)$ doesn't have any convergent subsequences. Consider the subsequences given by $y_n = x_{2n}$ and $z_n = x_{2n-1}$. Then the sets $\{y_n\}$ and $\{z_n\}$ are disjoint closed sets, and $d(y_n, z_n) \to 0$. This gives a contradiction by Lemma \ref{lem:const}. 
\end{proof}

\begin{proposition}\label{prop:UCderC}
If  $(M,d)$ is uniformly continuous, then $\derS(M)$ is compact. 
\end{proposition}
\begin{proof}
	The derived set $\derS(M)$ is closed for any topological space. 
    By Proposition \ref{prop:LS1}, $(M,d)$ is complete. 
    It thus suffices to show that $\derS(M)$ is totally bounded. 
  
  	Suppose not for contradiction. Then there exists $\varepsilon > 0$ such that $\derS(M)$ does not admit any finite cover by balls of radius $\varepsilon$. Hence there exists a sequence $(x_n)_{n \geq 2} \subset \derS(M)$ such that for any $n\neq m$, the pairwise distance $d(x_n, x_m) \geq \varepsilon$. As each $x_n\in \derS(M)$, there exists $y_n\in M$ such that $0 < d(x_n, y_n) < \varepsilon / n$. By triangle inequality 
    \[ d(y_n, y_m) \geq d(x_n, x_m) - d(x_n, y_n) - d(x_m, y_m) > \varepsilon (1 - \frac1n - \frac1m) > \frac\varepsilon6  \]
	when $n \neq m$ and $n, m \geq 2$. 
    
    The sets $\{y_n\}$ and $\{x_n\}$ are both uniformly isolated, and hence are both closed. The two sets are disjoint. However, by construction $\inf_{n} d(x_n, y_n) = 0$. This however contradicts Lemma \ref{lem:const}. Therefore we conclude that $\derS(M)$ must be totally bounded. 
\end{proof}
\begin{proposition}\label{prop:UCUI}
If $(M,d)$ is uniformly continuous, then for every $\varepsilon > 0$, $M \setminus B(\derS(M), \varepsilon)$ is uniformly isolated. 
\end{proposition}
\begin{proof}
	First recall that if $x \in M \setminus \derS(M)$, then by the definition of limit points there exists some $\delta_x$ such that $B(x,\delta_x) = \{x\}$. This means that the induced topology on $M \setminus \derS(M)$ is discrete. 
    
    We prove the proposition by contradiction. Suppose for some $\eta > 0$, the set $M \setminus B(\derS(M), \eta)$ is \emph{not} uniformly isolated. Observe that as $B(\derS(M), \eta)$ is open, $M \setminus B(\derS(M), \eta)$ is closed. And since the induced topology on it is discrete, \emph{any subset of $M \setminus B(\derS(M), \eta)$ is closed}. 
    As we assumed that $M \setminus B(\derS(M), \eta)$ is not uniformly isolated, we can find two disjoint minimizing sequences $(x_n)$ and $(y_n)$ such that $d(x_n, y_n) \to 0$. This again contradicts Lemma \ref{lem:const}. 
\end{proof}

We now have the following chain of implications
\begin{gather*}
	(M,d) \text{ is universally uniformly continuous}\\
    \Bigg\Downarrow \text{ \rlap{(by definition)}}\\
    (M,d) \text{ is uniformly continuous}\\
    \Bigg\Downarrow \text{ \rlap{(Propositions \ref{prop:UCderC} and \ref{prop:UCUI})}}\\
    \derS(M) \text{ is compact and } M \setminus B(\derS(M),\varepsilon) \text{ is uniformly isolated}\\
    \Bigg\Downarrow \text{ \rlap{(Proposition \ref{prop:derCUIUC})}}\\
    (M,d) \text{ is universally uniformly continuous}
\end{gather*}
and this proves both Theorem \ref{metricversion} and Corollary \ref{basicCor}.

\section{Proof of Theorem \ref{topologyversion}}

We start with a technical lemma. We say a triple of positive real numbers $\{a,b,c\}$ satisfies the triangle inequality if the inequalities
\[ a \leq b+ c, \quad b \leq a + c, \quad c \leq a + b \]
all hold. 
\begin{lemma}\label{lem:triangle}
If the triples $\{a,b,c\}$ and $\{x,y,z\}$ both satisfy the triangle inequality, then so does the triple $\{\max(a,x),\max(b,y),\max(c,z)\}$. 
\end{lemma}
\begin{proof}
If $a\leq b+c$ and $x\leq y+z$,then 
\[ \max(a,x) \leq \max(b+c, y+z) \leq \max(b,y) + \max(c,z). \]
The claim follows. 
\end{proof}

\begin{proof}[Proof of Theorem \ref{topologyversion}]
We give an explicit construction of the new metric $d$ from the old metric $\delta$. 
\begin{itemize}
	\item If $x = y$, set $d(x,y) = 0$. 
    \item If at least one of $x$ or $y$ is in $\derS(M)$, set $d(x,y) = \delta(x,y)$. 
    \item In the remaining cases, let $m$ be the unique integer such that 
    \[ 2^m < \max\Big(\delta\big(x,\derS(M)\big), \delta\big(y, \derS(M)\big)\Big) \leq 2^{m+1}. \]
    Set
    \[ d(x,y) = \max(\delta(x,y),2^m).\]
\end{itemize}

First we need to check that the function $d$ defined as above is a metric. It is obviously non-negative, non-degenerate, and symmetric. 
We only need to check that $d$ satisfies the triangle inequality.
This is where we used that $(M,\delta)$ is a metric space: by assumption $\delta$ satisfies the triangle inequality. 
We check on a case-by-case basis. 

Let $x,y,z$ be three distinct points in $M$. 
\begin{itemize}
	\item \emph{Case 1: at least two of the three points are in $\derS(M)$}\\
    	By definition the triple 
        \[ \big\{d(x,y), d(y,z), d(z,x)\big\} = \big\{\delta(x,y), \delta(y,z), \delta(z,x)\big\}\]
        since within each pair from $\{x,y,z\}$ at least one point is in $\derS(M)$. Then as $\delta$ satisfies the triangle inequality, so does $d$. 
	\item \emph{Case 2: exactly one of the three points is in $\derS(M)$}\\
    	Permuting the labels we can assume that $x\in \derS(M)$, that $\delta(y,\derS(M)) \in (2^n, 2^{n+1}]$, that $\delta(z,\derS(M)) \in (2^m, 2^{m+1}]$, and finally $n \leq m$. Then 
        \[ \Big\{d(x,y), d(x,z), d(y,z)\Big\} = \Big\{\delta(x,y), \delta(x,z), \max\big(\delta(y,z), 2^m\big)\Big\}.\]
        Observe now that the triple 
        $\{ 2^n, 2^m, 2^m \} $
        obviously satisfies the triangle inequality, and so does $\{\delta(x,y), \delta(x,z), \delta(y,z)\}$. We also have that since $x\in \derS(M)$, that $\delta(x,y) \geq \delta(y,\derS(M)) > 2^n$; and similarly for $\delta(x,z)$. And so we can write 
        \begin{multline*}
        	\Big\{d(x,y), d(x,z), d(y,z)\Big\} = \\
            \Big\{ \max\big(\delta(x,y),2^n\big), \max\big(\delta(x,z), 2^m\big), \max\big(\delta(y,z), 2^m\big)\Big\}
        \end{multline*}
        and conclude by Lemma \ref{lem:triangle} that the triangle inequality holds.
    \item \emph{Case 3: none of the three points are in $\derS(M)$}\\
		Define $m, n, p$ such that 
        \begin{gather*}
        	\delta(x,\derS(M)) \in (2^m, 2^{m+1}],\\
            \delta(y, \derS(M)) \in (2^n, 2^{n+1}],\\
            \delta(z, \derS(M)) \in (2^p, 2^{p+1}].\\
        \end{gather*}
		By permuting the labels we can assume $m \leq n \leq p$. Then the triple 
        \begin{multline*}
        \Big\{ d(x,y), d(y,z), d(x,z)\Big\} = \\
        \Big\{ \max\big(\delta(x,y), 2^n\big), \max\big(\delta(y,z), 2^m\big), \max\big(\delta(x,z), 2^m\big)\Big\}.
        \end{multline*}
        Arguing as the previous case, using that $\{2^n, 2^m, 2^m\}$ satisfies the triangle inequality, we see that again by Lemma \ref{lem:triangle} that the triangle inequality holds for $\{d(x,y), d(y,z), d(z,x)\}$. 
\end{itemize}

Thus we have shown that $d$ is a metric. 

Next we check that $(M,d)$ and $(M,\delta)$ have the same topology. 
This is true because they have essentially the same open balls:
\begin{itemize}
	\item If $x\in \derS(M)$, then $d(x,y) = \delta(x,y)$, and hence the balls $B_{\delta}(x,r) = B_d(x,r)$. 
    \item If $x\not\in \derS(M)$, then there exists some $r$ such that $B_{\delta}(x,r) = \{x\}$, since $x$ is not a limit point. We claim that $B_{d}(x,r) = \{x\}$ also, and hence $x$ is also isolated with respect to $d$. This claim holds because if $y\in \derS(M)$, then $d(x,y) = \delta(x,y) > r$; and if $y\not\in\derS(M)$ and $y \neq x$, we have $d(x,y) \geq \delta(x,y) > r$ by definition. 
\end{itemize}

Finally, we need to check that $(M,d)$ is uniformly continuous. We will use Theorem \ref{metricversion}. Since $(M,d)$ and $(M,\delta)$ have the same topology, $\derS(M)$ is still compact with respect to $d$. It suffices to check that $M \setminus B_d(\derS(M), \eta)$ is uniformly isolated. 
Let $n$ be such that $2^n < \eta \leq 2^{n+1}$.  
If $y \not\in B_d(\derS(M), \eta)$, using that $d(\derS(M), y) = \delta(\derS(M), y)$, we have that $\delta(\derS(M),y) > 2^n$. If $y_1, y_2 \in M \setminus B_d(\derS(M), \eta)$, then by the construction of $d$ we must have  
\[ d(y_1, y_2) \geq \max(\delta(y_1,y_2), 2^n) \geq 2^n. \]
And our theorem now follows.  
\end{proof}

\end{document}